\documentclass{amsart}
\usepackage{amsmath, amsthm, amssymb}
\usepackage{xspace}
\usepackage{enumitem}

\usepackage{mathtools}
\mathtoolsset{showonlyrefs}

\newtheorem{theorem}{Theorem}
\newtheorem{prop}[theorem]{Proposition}
\newtheorem{coro}[theorem]{Corollary}
\newtheorem{definition}[theorem]{Definition}
\newtheorem{lemma}[theorem]{Lemma}

\newcommand \R {{\mathbb{R}}}
\newcommand \B {{\mathbb{B}}}
\renewcommand \S {{S^{n-1}}}
\newcommand \Lowner {L\"owner\xspace}
\newcommand \M {{\operatorname{M}_n(\R)}}
\newcommand \GL {{\operatorname{GL}_n(\R)}}
\newcommand \SL {{\operatorname{SL}_n(\R)}}
\newcommand \sym {{\operatorname{Sym}}}
\newcommand \symp {{\operatorname{Sym_+}}}
\newcommand \vol {{\operatorname{vol}}}
\newcommand \tr {{\operatorname{tr}}}
\newcommand \I {{\operatorname{Id}}}
\newcommand \co {{\operatorname{co}}}
\newcommand \St {{\S \cap \partial K}}

\title[On explicit representations of isotropic measures]{On explicit representations of isotropic measures in John and L\"owner positions}
\author{F. M. Ba\^eta}
\author{J. Haddad}
\address{Departamento de Matem\'atica, ICEx,  Universidade Federal de Minas Gerais, 30.12370, Belo Horizonte, Brasil.}
\email{jhaddad@mat.ufmg.br, fernandah@ufmg.br}


\date{}

\begin{document}
\begin{abstract}
	Given a convex body $K \subseteq \R^n$ in \Lowner position we study the problem of constructing a non-negative centered isotropic measure supported in the contact points, whose existence is guaranteed by John's Theorem.
	The method we propose requires the minimization of a convex function defined in an $\frac {n(n+3)}2$ dimensional vector space.
	We find a geometric interpretation of the minimizer as $\left. \frac{\partial}{\partial r}(A_r, v_r)\right|_{r=1}$, where $A_r K + v_r$ is a one-parameter family of positions of $K$ that are in some sense related to the {\it maximal intersection position of radius $r$} defined recently by Artstein-Avidan and Katzin.
\end{abstract}
\maketitle

\section{Introduction}
Let $K \subset \R^n$ be a convex body (a compact convex set with non-empty interior).
In 1948 Fritz John \cite{john1948extremum} studied the problem of determining the ellipsoid $\mathcal E_J \subseteq K$ of maximal volume inside $K$ (known today as {\it Jonh's Ellipsoid}) and showed a set of necessary conditions for $\mathcal E_J$ to be the unit Euclidean ball $\B$ (Theorem \ref{thm_john} below).
A position of a convex body $K$ is a set of the form $A(K)$ where $A$ is an invertible affine transformation (a linear function composed with a translation).
We say that $K$ is in John position if $\mathcal E_J = \B$.
Since all ellipsoids are of the form $A(\B)$, there exists always a matrix $A_K \in \GL$ and vector $v_K \in \R^n$ for which $A_K K+v_K$ is in John position. 
A construction that is dual to John's ellipsoid is the {\it \Lowner ellipsoid} $\mathcal E_L \supseteq K$ which is the unique ellipsoid of minimal volume containing $K$.
The set $K$ is in \Lowner position if $\mathcal E_L = \B$.
John's Theorem can be stated as follows.
\begin{theorem}
\cite[Application 4, pag. 199 - 200]{john1948extremum}
	\label{thm_john}
	Assume $K$ is in John (resp. \Lowner) position, then there exists a finite set of points $\{\xi_1, \ldots, \xi_m\} \in S^{n-1} \cap \partial K$, positive numbers $\{c_1, \ldots, c_m\}$ and $\lambda \neq 0$, for which 
	\begin{equation}
		\label{eq_contactpoints}
		\sum_i c_i \xi_i \otimes \xi_i = \lambda\ \I \ \text{ and } \ \ \sum_i c_i \xi_i = 0.
	\end{equation}
Here $v \otimes w$ is the rank-one matrix $(v \otimes w)_{i,j} = v_i w_j$.
\end{theorem}
Since $\tr(\xi_i \otimes \xi_i) = |\xi_i|_2 = 1$, taking traces in the first equality of \eqref{eq_contactpoints} we obtain $\sum_i c_i = n \lambda$, determining the value of $\lambda$.

A measure $\mu$ on the sphere $S^{n-1}$ is said to be {\it isotropic} if
\[ \int_\S (\xi \otimes \xi) d\mu = \lambda\ \I,\]
for some $\lambda \neq 0$, and {\it centered} if
\[ \int_\S \xi d\mu = 0.\]
(Integration in $\R^n$ and $\M$ is understood to be coordinatewise).
Then one can see that equation \eqref{eq_contactpoints} can be expressed as the fact that the atomic measure $\mu_K = \sum_i c_i \delta_{\xi_i}$ is centered and isotropic.
Later Ball \cite{ball1992ellipsoids} proved that the existence of a non-negative centered isotropic measure $\mu_K$ in the set of contact points, guarantees that $K$ is in John position if $\B \subseteq K$, or in \Lowner position if $K \subseteq \B$.

The literature around the John/\Lowner position and its relation to isotropic measures, is vast.
The relation between extremal position and isotropic measures was studied extensively in \cite{giannopoulos1999isotropic,giannopoulos2000extremal,giannopoulos2001john}. Extensions to related minimization problems were studied in \cite{bastero2004positions, lutwak2005john, bastero2002john, gordon2004john, lasserre2015generalization}. Isotropic measures can also be used in combination with the Brascamp-Lieb inequality to find reverse isoperimetric inequalities, see \cite{ball1989volumes, ball1991volume, ball1992ellipsoids, lutwak2007volume}. 
A simple example of the relation between isotropic measures and extremal problems is the fact (see \cite{giannopoulos1999isotropic}) that a convex body $K$ is in position of minimal surface area measure (among positions of the same volume) if and only if its area measure is isotropic. We cite this example because in this case (as in many others) the isotropic measure is directly obtained from the convex body $K$. But the existence of the measure $\mu_K$ in Theorem \ref{thm_john} (and in particular, of $\{c_i, \xi_i\}$) is often shown in a non-constructive way, usually by proving first that it is impossible to separate $(\frac 1n \I,0) \in \M \times \R^n$ from the set 
\[\left\{(\xi \otimes \xi,\xi) \in \M \times \R^n / \xi \in S^{n-1} \cap \partial K\right\}\]
with linear functionals.
The same approach can be found in \cite[Section 2.1.3]{artstein2015asymptotic} and in \cite{gruber2005arithmetic, giannopoulos2000extremal}, the original idea of John \cite{john1948extremum} is different, but still non-constructive.
Recently in \cite{artstein2018isotropic}, Artstein and Katzin showed that $\mu_K$ can be constructed as a weak approximation of uniform measures on subsets of $\S$.
Moreover, they introduced a new one-parameter family of positions: A convex body $K$ is said to be in {\it maximal intersection position of radius $r$} if $r \B$ is the ellipsoid maximizing $\vol(r \B \cap K)$ among all ellipsoids of same volume as $r \B$.
It is also shown that every centrally symmetric convex body $K$ admits at least one of such positions $T_r K$ with $T_r \in \SL$, and in this case the uniform measure in $\S \cap r^{-1} T_r K$ is isotropic (modulo some technical assumptions).

For convenience we only formulate the theorem in the ``\Lowner version''.
\begin{theorem}[Theorem 1.6, \cite{artstein2018isotropic}]
	\label{thm_shiri}
	Let $K \subset \R^n$ be a centrally symmetric convex body. For every $r<1$, denote by $\nu_r$ the uniform probability measure on $\S \cap r^{-1} T_r K$, where $T_r K$ is in maximal intersection position of radius $r$.
	Then there exists a sequence $r_j \nearrow 1$ such that the sequence of measures $\nu_{r_j}$ weakly converges to an isotropic measure whose support is contained in $\partial K \cap \S$.
\end{theorem}

The isotropic measure is thus constructed (in the symmetric case), but one can argue that in practice there is a large body of computations to make. First of all, one has to solve a one-parameter family of minimization problems (find the matrix $T_r$ for $r$ close to $1$), and then take the limit of (some subsequence of) all these measures $\nu_r$.

The purpose of this paper is to present a simple finite dimensional minimization problem whose solution (when it exists) can be used to construct a non-negative centered isotropic measure as above.

We will assume always that $K$ is in \Lowner position.
Denote by $\sym, \symp \subseteq \M$ the sets of symmetric matrices and symmetric positive-definite matrices, respectively.
Also define $\sym_a = \{ A \in \sym / \tr(A) = a\}$ for $a \in \R$, and $\SL \subseteq \M$ the set of matrices of determinant one.

We shall prove the following:
\begin{theorem}
	\label{thm_isotropic_general}
	Let $K$ be a convex body in \Lowner position.
	Choose any finite positive and non-zero measure $\nu$ in $\St$, and any $C^1$ function $F:\R \to \R$ that is non-negative, non-decreasing, convex, strictly convex in $[0,\infty)$, and assume $F'(0) > 0$.
	Consider the convex functional $I_\nu:\sym \times \R^n \to \R$ defined by
	\[I_\nu (M,w) = \int_\S F(\langle \xi, M \xi + w \rangle) d\nu(\xi).\]
	If the restriction of $I_\nu$ to $\sym_0 \times \R^n$ has a unique global minimum $(M_0, w_0)$, then the measure 
	\[F'(\langle \xi, M_0 \xi + w_0 \rangle) d\nu(\xi)\]
	is non-negative, non-zero, centered and isotropic.
\end{theorem}

Let us consider the situation where $\St$ is finite. In this case, a natural choice of $\nu$ is the counting measure $c$.
\begin{coro}
	\label{cor_isotropic_finite}
	Let $K$ be a convex body in \Lowner position and assume 
	\[\St = \{\xi_1, \ldots, \xi_m\}.\]
	Choose any $C^1$ function $F:\R \to \R$ that is non-negative, non-decreasing, convex, strictly convex in $[0,\infty)$, and assume $F'(0) > 0$.
	Consider the convex functional $I_c:\sym \times \R^n \to \R$ defined by
	\[I_c(M,w) = \sum_{i=1}^m F(\langle \xi_i, M \xi_i + w \rangle).\]
	If the restriction of $I_c$ to $\sym_0 \times \R^n$ has a unique global minimum $(M_0, w_0)$, then the numbers
	\[c_i = F'(\langle \xi_i, M_0 \xi_i + w_0 \rangle), \ i=1,\ldots, m\]
	together with the vectors $\xi_i , \ i=1,\ldots, m$, satisfy equation \eqref{eq_contactpoints}.
\end{coro}

For any $K$ there is a canonical choice of measure $\nu$ given by 
\[\nu_K = C_0(\co(\St), \cdot),\]
where $C_0$ is the $0$-th curvature measure (see Section \ref{sec_preliminaries}), and will play a special role in Theorem \ref{thm_min_equivalences} below.
Depending on the set $\St$ and the measure $\nu$, the function $I_\nu$ might or might not have a minimum.
This can be a consequence of a ``bad choice'' of $\nu$, or of the fact that $\St$ is degenerate in some sense.
To make this precise we recall the following properties about John/\Lowner position.
A proof can be found for the symmetric case in \cite[proof of Theorem $2.1.10$ and Lemma $2.1.13$]{artstein2015asymptotic}.
In the general case, the proof is analogue.
\begin{theorem}
	\label{thm_lowner_equivalences}
	Let $L$ be any convex body.
	The following statements are equivalent
	\begin{enumerate}
		\item $L$ is in \Lowner position.
		\item $L \subseteq \B$ and for every $(M,w) \in (\sym_0 \times \R^n) \setminus (0,0)$ there exists $\xi \in \St$ for which $\langle \xi, M \xi + w \rangle \geq 0$.
		\item $L \subseteq \B$ and $(\frac 1n \I,0) \in \co\left( \{(\xi \otimes \xi, \xi) / \xi \in \St \} \right)$.
	\end{enumerate}
\end{theorem}

Recalling that $K$ is always assumed to be in \Lowner position, we can prove:
\begin{theorem}
	\label{thm_min_equivalences}
	The following statements are equivalent
	\begin{enumerate}
		\item The restriction of $I_\nu$ to $\sym_0 \times \R^n$ has a unique global minimum $(M_0, w_0)$.
		\item For every $(M,w) \in (\sym_0 \times \R^n) \setminus (0,0)$ 
			\[\nu(\{\xi \in \St / \langle \xi, M \xi + w \rangle > 0\}) > 0.\]

If $\nu = \nu_K$ or if $\St$ is finite and $\nu = c$, the statements above are also equivalent to the following:

		\item $(\frac 1n \I,0)$ lies in the interior of $\co\left( \{(\xi \otimes \xi, \xi) / \xi \in \St \} \right) \subseteq \sym_1 \times \R^n$, where the interior is taken with respect to $\sym_1 \times \R^n$.
	\end{enumerate}
\end{theorem}

Notice that these conditions do not depend on the choice of $F$ but only on the measure $\nu$.

Even though Theorem \ref{thm_isotropic_general} appears to have no relation to Theorem \ref{thm_shiri}, both of them stem from a construction that we present next.
For any $r \in (1/2,1)$ and any function $h:\R \to \R$ define 
\[h_r(s) = h\left(\frac{s-1}{1-r}\right).\]
Let us fix two measurable functions $f,g:\R \to \R$.
We define the functional $L_r:\M \times \R^n \to \R$ by
\[L_r(A,v) = \frac 1{1-r} \int_{\R^n} f_r(|A x + v|_2) g_r(\|x\|_{K}) dx,\]
and the functional $I_r:B_r \times \R^n \subseteq \sym \times \R^n \to \R$ by
\begin{align}
	I_r(M, w) 
	&= \frac 1{1-r} \int_{\R^n} f_r(|x|_2) g_r\left(\left\|(\I + (1-r) M)^{-1} (x - (1-r)w)\right\|_{K}\right) dx\\
	&= \frac 1{1-r} \int_{\R^n} f_r(|x|_2) g_r\left(\left\|x - (1-r)w \right\|_{(\I+(1-r)M)K}\right) dx.
\end{align}
The domain $B_r$ is the set of matrices $M$ such that $\I+(1-r)M$ is invertible, and in particular it contains the ball $B(0, (1-r)^{-1})$ in the operator norm.
It is important to notice that $L_r(A,v) = I_r\left(\frac{A-\I}{1-r},\frac v{1-r}\right)$ for $(A,v) \in \SL \times \R^n$.
Taking $f = 1_{[-1,\infty)}, g = 1_{(-\infty, 0]}$ one obtains for $(A,v) \in \SL \times \R^n$,
\[L_r(A,v) = \frac 1{1-r} \vol((A K+v) \setminus r \B)\]
meaning that a minimum $(A_r,v_r)$ of the restriction of $L_r$ to $(\SL \cap \sym_+) \times \R^n$ induces a maximal intersection position of radius $r$.

The functional $L_r$ satisfies the invariance property $L_r(O A, O v) = L_r(A,v)$ for every orthogonal matrix $O$.
By polar decomposition, it suffices to know the behaviour of $L_r$ restricted to $\symp \times \R^n$.
For example, a global minimum of the restriction of $L_r$ to $\symp \times \R^n$ is also a global minimum in $\M \times \R^n$.

If the function $f$ is smooth then $L_r$ becomes a smooth functional (in Section 2 this statement is made precise).
Let $(A_r, v_r)$ be a global minimum of the restriction of $L_r$ to $(\symp \cap \SL) \times \R^n$,
and hence, a global minimum of the restriction of $L_r$ to the smooth hypersurface $\SL \times \R^n \subseteq \M \times \R^n$.
Considering the Frobenius inner product in $\M$ given by 
\[\langle A, B \rangle = \tr(A^T B) = \sum_{i,j} A_{i,j} B_{i,j},\] 
and the product in $\M \times \R^n$ given by 
\begin{equation}
	\label{eq_innerproduct}
	\langle (A,v), (B,w) \rangle = \langle A, B \rangle + \langle v, w\rangle,
\end{equation}
we can write the equation of Lagrange multipliers $ \lambda \nabla \det(A_r) = \nabla L_r(A_r, v_r) $ (the approach of Lagrange multipliers was used already in \cite{bastero2006dual}) and we get
\begin{align}
	\lambda& (A_r^{-T}, 0)
	= \frac 1{1-r} \int_{\R^n} (f_r)'(|A_r x + v_r|_2) g_r(\|x\|_K) \left( \frac{A_r x + v_r }{|A_r x + v_r|_2} \otimes x, \frac{A_r x + v_r }{|A_r x + v_r|_2}\right) dx \\
	&= \frac 1{(1-r)^2} \int_{\R^n} (f')_r(|x|_2) g_r(\|A_r^{-1}(x-v_r)\|_K) \left( \left( \frac{x}{|x|_2} \otimes A_r^{-1}(x-v_r) \right), \frac{x}{|x|_2} \right) dx \\
	&= \frac 1{(1-r)^2} \int_{\R^n} (f')_r(|x|_2) g_r(\|A_r^{-1}(x-v_r)\|_K) \left( \left( \frac{x}{|x|_2} \otimes x - \frac{x}{|x|_2} \otimes v_r \right)A_r^{-T} , \frac{x}{|x|_2} \right) dx 
\end{align}
which implies
\begin{align}
	\label{eq_isotropic_Rn}
	\frac 1{1-r}\int_{\R^n} \frac{ (f')_r(|x|_2) }{|x|_2} g_r(\|A_r^{-1}(x-v_r)\|_K) \left( x \otimes x \right) dx &= \lambda_r \I \\
	\frac 1{1-r}\int_{\R^n} \frac{ (f')_r(|x|_2) }{|x|_2} g_r(\|A_r^{-1}(x-v_r)\|_K) x dx &= 0\\
\end{align}
(here the gradients are taken with respect to the whole space $\M \times \R^n$).

Under some conditions of $f,g$ the measure $\frac 1{1-r} (f')_r(|x|_2) g_r(\|A_r^{-1}(x-v_r)\|_K) dx$ concentrates near $\St$ as $r \to 1^-$ and converges for some sequence $r_k \to 1^-$ to a centered isotropic measure, as in Theorem \ref{thm_shiri}.
In connection with Theorem \ref{thm_shiri}, informally speaking, if $f = 1_{[-1, \infty)},g = 1_{(-\infty, 0]}$ as before, then $(f')_r$ has a Dirac delta at $r$, and the integrals above can be interpreted in the distributional sense, as integrals over $r \S$
\begin{align}
	\int_{A_r K+v_r} \frac{(f')_r(|x|_2)}{|x|_2} \left( x \otimes x\right) dx &= \frac 1r \int_{(A_r K + v_r) \cap r \S} \left( \xi \otimes \xi\right) d\xi = \tilde\lambda_r \I\\
	\int_{A_r K+v_r} \frac{(f')_r(|x|_2)}{|x|_2} x dx &= \frac 1r \int_{(A_r K + v_r) \cap r \S} \xi d\xi = 0,\\
\end{align}
thus obtaining, at least heuristically, the isotropic measure in Theorem \ref{thm_shiri}.
To formalize this result we need an approximation argument.
This technical computation is done in \cite[Theorem 2.6]{artstein2018isotropic} and it will not be reproduced in this paper.

Unfortunately, this choice of $f$ and $g$ do not give $L_r$ the desirable properties to work with critical point theory.
For instance, it is not known in general if the maximal intersection position of radius $r$ is unique, because $L_r$ is not convex (see \cite[Section 4]{artstein2018isotropic} and \cite{bochi2018approximation} for discussions about the uniqueness problem).
Instead, we will assume the following properties for $f$ and $g$.
\begin{enumerate}[label={\bf f\arabic*}, ref = {\bf f\arabic*}]
	\item $f$ is locally Lipschitz. \label{f_smooth} \label{fg_first}
	\item $f$ is convex. \label{f_convex}
	\item $f(x) = 0$ for $x \leq -1$. \label{f_zeroleft}
	\item $f$ is strictly increasing in $[-1, \infty)$. \label{f_inc}
\end{enumerate}
\begin{enumerate}[label={\bf g\arabic*}, ref = {\bf g\arabic*}]
	\item $g$ is locally Lipschitz. \label{g_smooth}
	\item $g$ is non-increasing. \label{g_decpos}
	\item $g(x) = 1$ for $x \leq -1$. \label{g_1}
	\item $g(x) > 0$ for $x \in (-1,1)$. \label{g_pos}
	\item $g(x) = 0$ for $x \geq 1$. \label{g_0} \label{fg_last}
\end{enumerate}
Two functions satisfying \ref{fg_first} to \ref{fg_last} are
\[
	f(x) = 
		\begin{cases}
			0 ,& \text{ if } x \leq -1\\
			x+1 ,& \text{ if } x>-1
		\end{cases}
	,\quad
	g(x) = 
		\begin{cases}
			1 ,& \text{ if } x \leq -1\\
			\frac {1-x}2 ,&\text{ if } x \in (-1,1)\\
			0 ,& \text{ if } x \geq 1
		\end{cases}.
\]
\begin{figure}[t]
	\centering
	\begin{minipage}{0.49 \textwidth}
		\centering
		\includegraphics[width=.9 \textwidth]{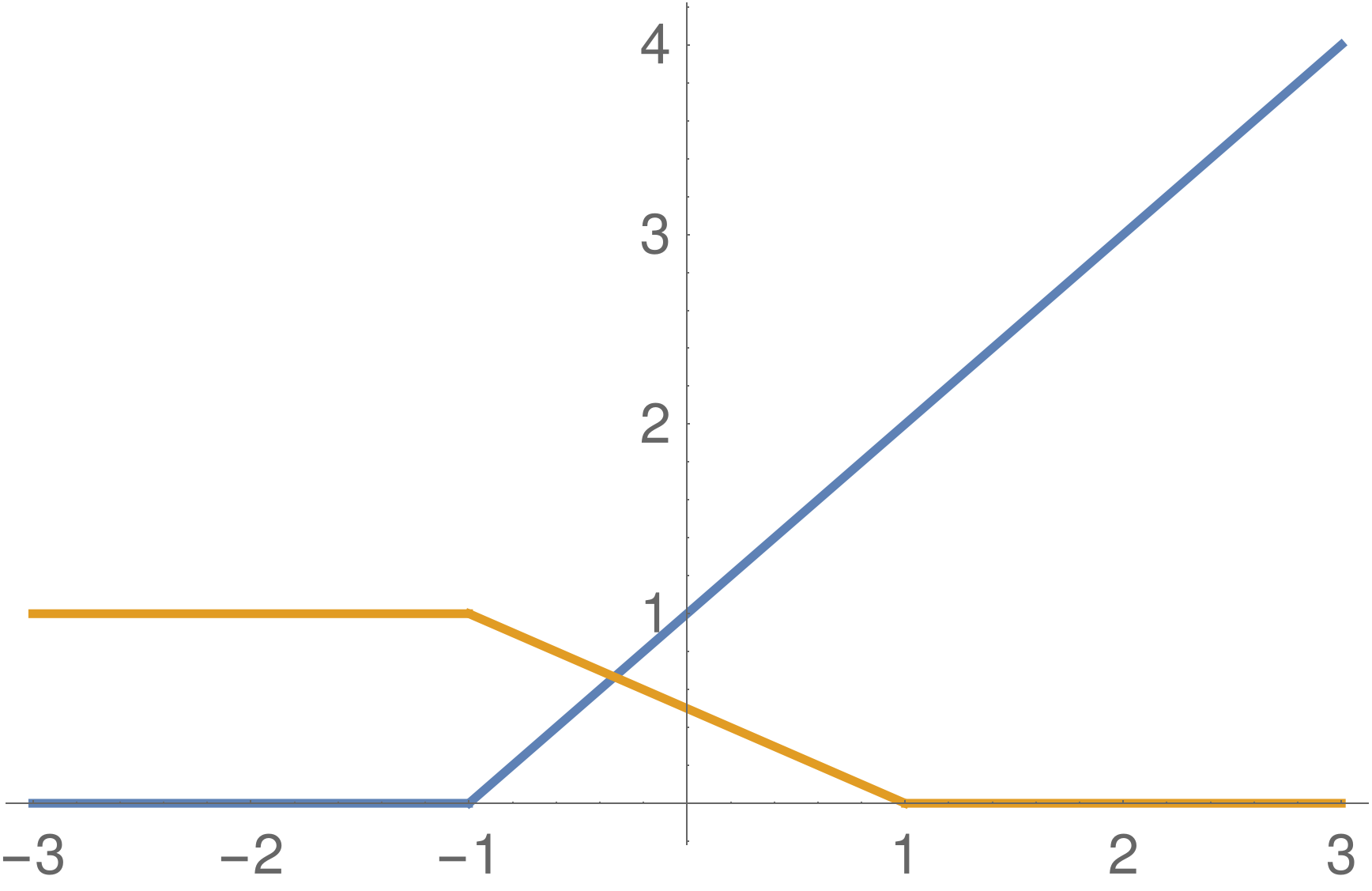}
	\caption{ functions $f$ (blue) and $g$ (orange),}
	\end{minipage}
	\hfill
	\begin{minipage}{0.49 \textwidth}
		\centering
	\includegraphics[width=.9 \textwidth]{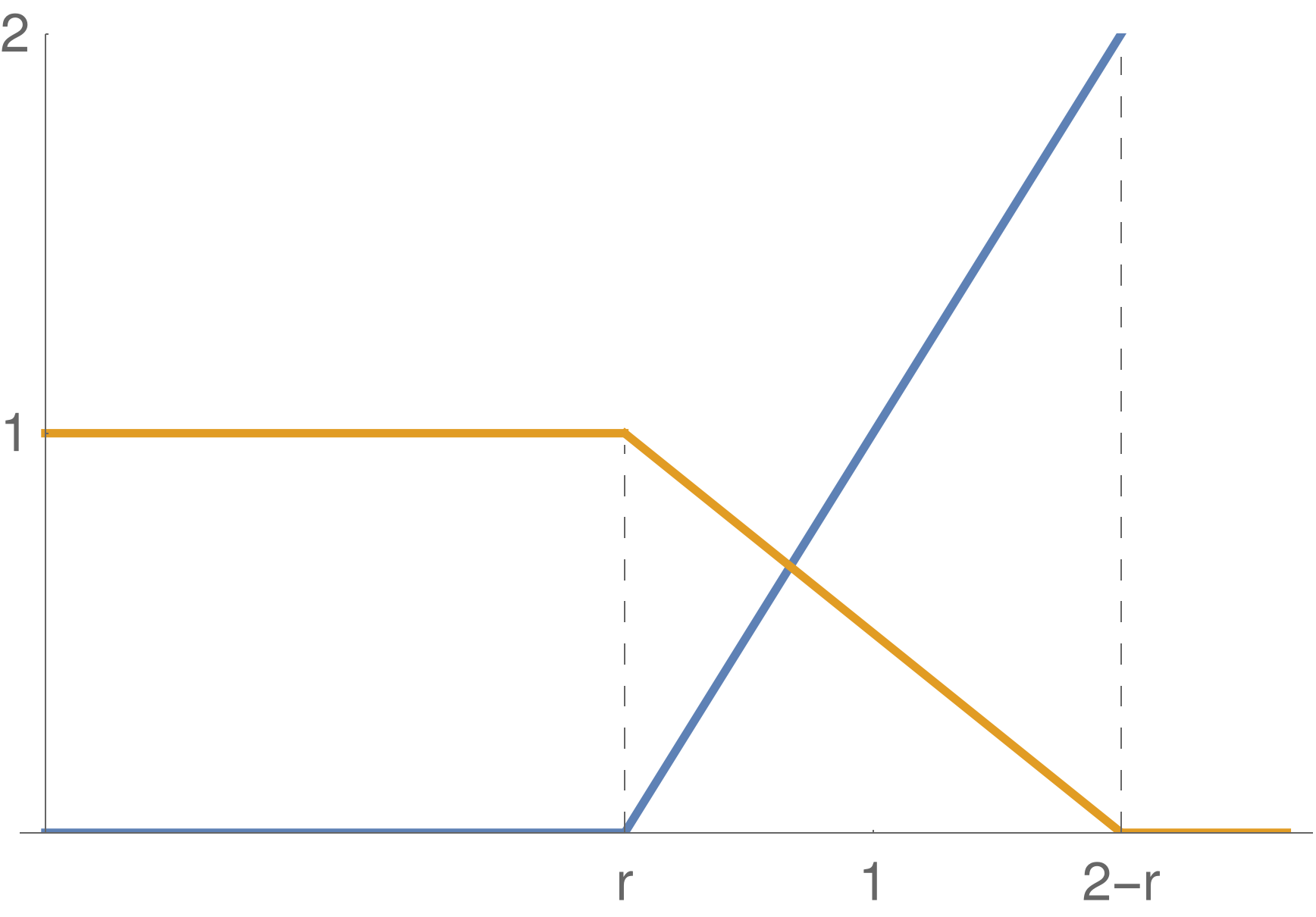}
		\caption{ functions $f_r$ (blue) and $g_r$ (orange) }
	\end{minipage}
\end{figure}
These conditions (specially \ref{f_convex}) will guarantee that $L_r$ has a unique global minimum $(A_r,v_r) \in (\SL \cap \sym_+) \times \R^n$ for $r \in (1/2,1)$.
In Section 3 we shall prove:
\begin{theorem}
	\label{thm_min_existence}
	Let $K$ be a convex body in \Lowner position and let $f,g$ satisfy all the properties \ref{fg_first} to \ref{fg_last}, then for every $r \in (1/2,1)$ the restriction of $L_r$ to $(\symp \cap \SL) \times \R^n$ has a unique minimum $(A_r, v_r)$ with $\lim_{r \to 1^-} (A_r, v_r) = (\I, 0)$.
	Likewise, the restriction of $I_r$ to 
	\[ \left( \frac{(\SL\cap \sym_+) -\I}{1-r} \right) \times \R^n\]
	has the unique minimum $(M_r, w_r) = (\frac{A_r-\I}{1-r}, \frac{v_r}{1-r})$
	with $\tr\left(\frac{M_r}{\|M_r\|_F} \right) \to 0$ as $r \to 1^-$.
\end{theorem}
Observe that as $r \to 1^-$, the matrices $M_r$ get close to the tangent space $T_\I (\SL \cap \sym) = \sym_0$.
The advantage of working with the functional $I_r$, and the connection with Theorem \ref{thm_isotropic_general} are made clear by the following theorem:
\begin{theorem}
	\label{thm_r1extension}
	Assume $K$ has a $C^1$-smooth boundary and all the properties \ref{fg_first} to \ref{fg_last} are satisfied.
	The functional $I_r(M,v)$ is extends continuously to $r=1$ as
	\[
		I_1(M, v) = \int_\St F(\langle \xi, M \xi + v\rangle ) d\xi,
	\]
	where $F$ is the convolution $F(x) = f * \bar g(x), \bar g(x) = g(-x)$ and satisfies the conditions of Theorem \ref{thm_isotropic_general}.
	Moreover, $I_r \to I_1$ as $r \to 1^-$, uniformly in compact sets.
\end{theorem}
Thus we obtain $I_1 = I_{\sigma}$ where $\sigma$ is the $n-1$ dimensional Hausdorff measure restricted to $\St$.
As in the remark after Theorem \ref{thm_shiri} with respect to the measure $\nu_r$, the measure $(f')_r(|x|_2) g_r(\|A_r^{-1}(x-v_r)\|_K) dx$ requires the computation of the matrix $A_r$.
The reason why Theorem \ref{thm_isotropic_general} is direct, is the following.
\begin{theorem}
	\label{thm_ADerivative}
	Assume all the properties \ref{fg_first} to \ref{fg_last} are satisfied and the function $I_1$ restricted to $\sym_0 \times \R^n$ has a unique global minimum $(M_0, w_0)$, then $\left. \frac{\partial (A_r, v_r)}{\partial r} \right|_{r=1}$ exists and is equal to $-(M_0, w_0)$.

	In this case, if $(\bar A_r, \bar v_r)$ is any curve in $\symp \times \R^n$ of the form
	\[(\bar A_r, \bar v_r) = (\I,0) + (1-r) (M_0, w_0) + o(1-r),\]
	the measure
	\[\frac 1{1-r} \frac{(f')_r(|x|_2)}{|x|_2} g_r(\|\bar A_r^{-1}(x-\bar v_r)\|_K) dx\]
	converges weakly to $F'(\langle \xi, M_0 \xi + w_0 \rangle) d\sigma$.
	
	For example, this is true for the linear part $\bar A_r = \I + (1-r) M_0, \bar v_r = (1-r) w_0$.
\end{theorem}

The paper is organized as follows: In Section \ref{sec_preliminaries} we recall the basic notations and properties we need about Convex Geometry, in Section \ref{sec_basic_results} we prove some basic technical properties of the functions $L_r, I_r$, and in Section \ref{sec_mainthms} we prove the main theorems which are Theorems \ref{thm_isotropic_general}, \ref{thm_min_equivalences}, \ref{thm_min_existence}, \ref{thm_r1extension} and \ref{thm_ADerivative}.

\section{Notation and preliminary results}
\label{sec_preliminaries}
We refer to the books \cite{artstein2015asymptotic, schneider2014convex} for basic facts from convexity.
Recall that a convex body is a compact convex set with non-empty interior.

Denote by $|x|_2$ the euclidean norm of $x \in \R^n$ and by $\|x\|_K$ the gauge function of $K$ defined by
\[\|x\|_K = \inf\left\{\lambda > 0 / x \in \lambda K \right\}.\]
The function $\|\cdot\|_K$ is well defined for every $x$ when $0$ is in the interior of $K$, and is a norm when $K$ is a symmetric convex body (this is $K = -K$).
Observe that $\|A^{-1} x\|_K = \|x\|_{A K}$ and that $\|A^{-1}(x-v)\|_K \leq 1$ if and only if $x \in A K + v$, for any $A \in \GL$ and $v \in \R^n$.

We say that $K$ is $C^k$ or that it has a $C^k$-smooth boundary if $\|x\|_K$ is a $C^k$ function in $\R^n \setminus \{0\}$.
If $K$ is $C^1$, for every $x \in \partial K$ there is a well-defined unit vector $n^K(x)$ orthogonal to $\partial K$ at $x$, pointing outwards $K$.
We extend the definition of $n^K$ to $\R^n \setminus \{0\}$ by homogeneity, this is, $n^K(\lambda x) = n^K(x)$ for every $\lambda > 0$.
Then the gradient of $\|\cdot\|_K$ at any point $x \in \R^n$ can be computed as
\begin{equation}
	\label{eq_normgradient}
	\nabla \|x\|_K = \frac{ \|x\|_K}{\langle n^K(x), x \rangle} n^K(x)
\end{equation}
(see for example \cite[Section 1.7.2, eq. (1.39)]{schneider2014convex}).

For $A \in \M$ the operator norm and the Frobenius norm are defined by
\[\|A\|_{op} = \sup_{|v|_2=1} |A v|_2,\ \  \|A\|_F = \sqrt{\tr(A^T A)},\]
respectively.
For for $(A,v) \in \M \times \R^n$, we use $\|(A,v)\| = \sqrt{\|A\|_F^2 + |v|^2}$ which is the norm induced by the inner product \eqref{eq_innerproduct}.

\begin{definition}
	We say that a function $f: U \subseteq \R^d \to \R$ is coercive if $\lim_{|x|_2 \to \infty} f(x) = + \infty$.
	A one-parameter family of functions $f_r: U \subseteq \R^d \to \R$ is said to be coercive uniformly in $r$ if $f_r(x) \geq f(x)$ for every $|x| \geq C$ for some $C > 0$ and some coercive function $f$.

	For a function $f:\R \to \R$ we say that $f$ is coercive to the right if \\ $\lim_{x \to +\infty} f(x) = + \infty$.
\end{definition}
	Observe that a coercive and convex function defined in a convex set must have points of minimum, and that if the convexity of the function is strict the minimum is unique.

We consider the set $\mathcal D = \{M \in \symp/ \det(M) \geq 1\}$.
The inequality
\[\det(\lambda A + (1-\lambda) B)^{1/n} \geq \lambda \det(A)^{1/n} + (1-\lambda) \det(B)^{1/n}\]
valid for all $A,B \in \symp$ (see for example \cite[Lemma 2.1.5,~p.51]{artstein2015asymptotic}) shows that $\mathcal D \subseteq \sym_+$ is a convex set.

Throughout this paper we assume that $K \subseteq \R^n$ is a fixed convex body in \Lowner position.
A basic fact that we shall use often is that $0$ is in the interior of $K$. If this where not true, there would exist $v \in \S$ such that $\langle v, x \rangle \geq 0$ for every $x \in K$.
Taking inner product against $v$ in the second equality of \eqref{eq_contactpoints}, we obtain
\[\sum c_i \langle v, u_i \rangle = 0\]
which implies $c_i = 0$ for all $i$.

For a convex body $L$, the $i$-th curvature measures are defined by means of local parallel sets (cf. \cite[Section 4 and formula (4.10)]{schneider2014convex}).
The $0$-th curvature measure $C_0(L, \cdot)$ can be viewed as a non-smooth extension of the Gauss-Kronecker curvature since by \cite[(4.25)]{schneider2014convex}, if $L$ is $C^2$-smooth, we have
\[C_0(L, \beta) = \int_{\partial L \cap \beta} \kappa(x) d S(x).\]
We shall need the following property.
\begin{prop}{\cite[Theorem 4.5.1]{schneider2014convex}}
\label{thm_Csupport}

	The measure $C_0(\co(\St), \cdot)$ is supported in $\St$.
\end{prop}

\section{Basic Results}
\label{sec_basic_results}

This section is devoted to establish basic properties of the functionals $I_r, L_r$.
For convenience, in each proposition we detail the properties of $f,g$ that are necessary.

The following fact is an easy consequence of Rademacher's Theorem and the Dominated Convergence Theorem.
The proof is straightforward and will be omitted.
\begin{prop}
	Assume \ref{f_smooth}, \ref{g_smooth}, \ref{g_0} are satisfied, then $L_r, I_1$ are $C^1$ for $r \in (1/2, 1)$.
\end{prop}

\begin{prop}
	\label{prop_Lcoercive}
	Assume \ref{f_convex}, \ref{f_zeroleft}, \ref{f_inc}, \ref{g_1}, then the family of functionals $L_r$ restricted to $\mathcal D \times \R^n$ is coercive, uniformly for $r \in (1/2,1)$.
\end{prop}
\begin{proof}
	Take any $(A,v) \in \mathcal D \times \R^n$.
	Let $B \subseteq \frac 12 K$ be a centered ball.
	Let $w$ be an eigenvector of $A$ of eigenvalue $\|A\|_{op}$, with euclidean norm $\frac 12 \|A\|_{op}$ and chosen so that $\langle v, w \rangle \geq 0$.
	Consider the half-space $S = \{x \in \R^n / \langle x, v+w \rangle \geq \langle v+w, v+w \rangle \}$ where clearly $v+w \in \partial S$.
	Applying the inverse of the affine transformation,
	\[\vol((AB+v) \cap S) = \det(A) \vol(B \cap \bar{S})\]
	where $\bar S = A^{-1}(S - v)$ is a half-space with $\frac 12 \frac{w}{|w|_2} \in \partial \bar S$.
	Since the volume of the intersection of $B$ with a half-space (not containing the origin) is a decreasing function of the distance of this half-space to the origin, and since this distance is at least $1/2$, we have $\vol(B \cap \bar{S}) \geq C_n$ where $C_n > 0$ is a dimensional constant.
	Also, $|x|_2 \geq |v+w|_2 \geq \sqrt{|v|_2^2 + |w|_2^2}$ for every $x \in S$.
	Using \ref{g_1} and that $f$ is non-negative and non-decreasing,
	\begin{align}
		L_r(A,v)
		&\geq \frac1{1-r} \int_B f_r(|Ax+v|_2) dx\\
		&\geq 2 \det(A)^{-1} \int_{(AB+v) \cap S} f_r(|x|_2) dx\\
		&\geq 2 \det(A)^{-1} \vol((AB+v) \cap S) f_r\left(\sqrt{|v|_2^2 + \frac 14 \|A\|_{op}^2}\right)\\
		&\geq 2 C_n f\left(\frac{\sqrt{|v|_2^2 + \frac 14 \|A\|_{op}^2}-1}{1-r}\right).\\
	\end{align}
	For $\sqrt{|v|_2^2 + \frac 14 \|A\|_{op}^2} \geq 1$ we obtain
	\[
		L_r(A,v) \geq C_n f\left(\sqrt{|v|_2^2 + \frac 14 \|A\|_{op}^2}-1\right)\\
	\]
	and $f$ is coercive to the right, consequence of \ref{f_convex} and \ref{f_inc}.
\end{proof}

\begin{prop}
	\label{prop_Lposconvex}
	Let $r \in (1/2, 1)$ and assume \ref{g_1}, \ref{g_pos}, \ref{f_convex}, \ref{f_zeroleft}, \ref{f_inc}.
	The function $L_r$ restricted to $\mathcal D \times \R^n$ is positive and convex.
\end{prop}
\begin{proof}
	{\bf Positive:}

	Take $(A,v) \in \mathcal D \times \R^n$ and assume $L_r(A,v)=0$.
	Since the functions $f_r, g_r$ are non-negative we have
	\[g_r(\|x\|_K) > 0 \Rightarrow f_r(|A x+v|_2) = 0\]
	meaning
	\[x \in (2-r) K \Rightarrow A x+v \in r \B,\]
	hence,
	\[(2-r) A.K+v\subseteq r \B.\]

	Since $K$ is in \Lowner position,
	\[\det\left(\frac{2-r}{r} A \right) \leq 1\]
	\[\det(A) \leq \left(\frac{r}{2-r}\right)^n < 1\]
	which is absurd.
		
	{\bf Convex:}
	Let $(A,v), (B,w) \in \mathcal D \times \R^n$, and $\lambda \in [0,1]$.
	By \ref{f_zeroleft}, \ref{f_inc}, $f$ is non-decreasing.
	Using this fact, the convexity of the Euclidean norm and \ref{f_convex},
	\begin{align}
		\label{eq_Lposconvex_convex_ineq}
		L_r(\lambda A & + (1-\lambda)B, \lambda v+(1-\lambda)w)\\
		&= \frac 1{1-r} \int_{\R^n} f_r(|(\lambda A+(1-\lambda)B)x+(\lambda v+(1-\lambda)w)|_2) g_r(\|x\|_K) dx\\
		&= \frac 1{1-r} \int_{\R^n} f_r(|\lambda (Ax+v)+(1-\lambda)(Bx+w)|_2) g_r(\|x\|_K) dx\\
		&\leq \frac 1{1-r} \int_{\R^n} f_r(\lambda |Ax+v|_2+(1-\lambda)|Bx+w|_2) g_r(\|x\|_K) dx \label{ineq1}\\
		&\leq \frac 1{1-r} \int_{\R^n} (\lambda f_r(|Ax+v|_2)+(1-\lambda)f_r(|Bx+w|_2)) g_r(\|x\|_K) dx \label{ineq2}\\
		&= \lambda L_r(A,v) + (1-\lambda) L_r(B,w).
	\end{align}
\end{proof}
\begin{prop}
	\label{prop_LIdbounded}
	Assume \ref{g_0}, \ref{f_zeroleft}, then for $r \in (1/2, 1)$ we have $L_r(\I, 0) \leq C$ where $C$ is a constant depending only on $f$.
\end{prop}
\begin{proof}
	Use \ref{g_0}, polar coordinates, the substitution $s = 1+(1-r)t$ and \ref{f_zeroleft} to obtain
	\begin{align}
		L_r(\I,0)
		&\leq \frac 1{1-r} \int_{(2-r) \B} f_r(|x|_2) dx\\
		&\leq \frac 1{1-r} \int_\S \int_0^{2-r} f_r(s) ds d\xi\\
		&= \int_\S \int_{-\infty}^1 f(t) dt d\xi\\
		&= \int_\S \int_{-1}^1 f(t) dt d\xi
		\leq C
	\end{align}
	(notice that $f_r(s) = f(t)$).
\end{proof}

To prove that $L_r$ has at most one minimum we need a lemma.

\begin{lemma}
	\label{lem_affine_multiplies}
	Let $n\geq 2$, $A,B \in \GL$ and $v,w \in \R^n$ be such that $Ax+v$ is a multiple of $Bx+w$ for every $x$ in an open set $U \subseteq \R^n$.
	Then there exists $a \neq 0$ for which $A = aB$ and $v = a w$.
\end{lemma}
\begin{proof}
	Without loss of generality, we may assume that $Ax+v\neq 0 \neq Bx+w$ in $U$, by eventually excluding from $U$ the points $-A^{-1}v$ and $-B^{-1}w$.
	There is a function $a(x)$ such that
	\begin{equation}
		\label{eq_affine_multiples_1}
		Ax+v = a(x)(Bx+w).
	\end{equation}
	Take any $x_0 \in U$ and choose a coordinate $i$ for which $(B x_0 + w)_i \neq 0$.
	The formula
	\[a(x) = \frac{(Ax+v)_i}{(Bx+w)_i}\]
	guarantees that $a(x)$ is a $C^\infty$ function near $x_0$.
	
	Take the directional derivative of equation \eqref{eq_affine_multiples_1} with respect to $x$, in the direction of a vector $x_1$
	\[ A x_1 = \langle \nabla a(x), x_1 \rangle (Bx+w) + a(x) B x_1. \]
	For $x=x_0$, equality for every $x_1$ implies
	\begin{equation}
		\label{eq_affine_multiples_2}
		A = (B x_0 + w) \otimes \nabla a(x_0) + a(x_0) B.
	\end{equation}
	Now take in \eqref{eq_affine_multiples_2}, the directional derivative with respect to $x$, in the direction of $x_2$, at $x=x_0$
	\[
		0 = (x_1^T . Ha(x_0) . x_2) (B x_0 + w) + \langle \nabla a(x_0), x_1 \rangle B x_2 + \langle \nabla a(x_0), x_2 \rangle B x_1.
	\]

	We claim that $\nabla a(x_0) = 0$.
	If $B \nabla a(x_0)$ is parallel to $B x_0 + w$ take $x_1 = \nabla a(x_0)$ and $x_2$ orthogonal to $\nabla a(x_0)$ to obtain
	\[
		0 = (\nabla a(x_0)^T . Ha(x_0) . x_2) (B x_0 + w) + \langle \nabla a(x_0), \nabla a(x_0) \rangle B x_2,
	\]
	which implies $\nabla a(x_0) = 0$ because $B x_2$ is not parallel to $B \nabla a(x_0)$.
	If $B \nabla a(x_0)$ is not parallel to $B x_0 + w$ take $x_1 = x_2 = \nabla a(x_0)$
	\[
		0 = (\nabla a(x_0)^T . Ha(x_0) . \nabla a(x_0)) (B x_0 + w) + \langle \nabla a(x_0), \nabla a(x_0) \rangle B \nabla a(x_0),
	\]
	which implies $B \nabla a(x_0) = 0$ and we conclude that $\nabla a(x_0) = 0$.

	Formula \eqref{eq_affine_multiples_2} becomes $A = a(x_0) B$, and by \eqref{eq_affine_multiples_1} for $x=x_0$ we get
	\[a(x_0) B x_0 + v = a(x_0) (B x_0 + w)\]
	which implies $v = a(x_0) w$.

	Finally $A \in \GL$ implies $a = a(x_0) \neq 0$.
	The proof is complete.
\end{proof}

\section{Main Results}
\label{sec_mainthms}

\begin{proof}[Proof of Theorem \ref{thm_min_existence}]
	First we show that for $r \in (1/2, 1)$, the restriction of $L_r$ to $\mathcal D \times \R^n$ has exactly one global minimum $(A_r, v_r)$.
	By coercivity (Proposition \ref{prop_Lcoercive}) there is at least one minimum.
	Assume the minimum is attained in two different points $(A,v)$ and $(B,w)$.
	Since $L_r$ is convex (Proposition \ref{prop_Lposconvex}) there is equality in equation \eqref{ineq1} for every $\lambda \in [0,1]$, then since $f_r$ is strictly increasing in $[r, \infty)$,
	$Ax+v$ and $Bx+w$ are multiples for every $x \in (2-r)K$ such that $|Ax+v|_2 > r$ or $|Bx+w|_2 > r$.
	Since $((2-r)A K+v) \setminus r \B$ has non-empty interior, Lemma \ref{lem_affine_multiplies} implies
	\[A = a B, v = a w\]
	for some number $a > 0$. If $a=1$ we are done.
	Assume without loss of generality that $a > 1$.
	Since $(A K + v) \setminus r \B$ has non-empty interior (because $K$ is in \Lowner position and $\det(A) = a \det(B) > 1$) we have
	\[L_r(B,w) < L_r(A,v),\]
	which is absurd.

	Now we show that $A_r \in \SL$.
	If $\det(A_r) > 1$ then again $(A_r K + v_r) \setminus r \B$ has non-empty interior and
	\begin{align}
		L_r(\det(A_r)^{-1/n} A_r & , \det(A_r)^{-1/n} v_r)\\
		&= \frac 1{1-r} \int_{\R^n} f_r(|\det(A_r)^{-1/n} (A_r x+v_r)|_2) g_r(\|x\|_K) dx\\
		&< \frac 1{1-r} \int_{\R^n} f_r(|A_r x + v_r|_2) g_r(\|x\|_K) dx\\
		&=L_r(A_r, v_r),
	\end{align}
	which is absurd.
	We conclude then that $A_r \in \partial \mathcal D = \SL \cap \symp$.

	Denote $M_r = \frac{A_r -\I}{1-r}, w_r = \frac{v_r}{1-r}$.
	Since $A_r \in \SL$, we have $I_r(M_r,w_r) = L_r(A_r, v_r)$ and $(M_r, w_r)$ is the unique global minimum of the restriction of $I_r$ to
	\[ \left( \frac{(\SL\cap \sym_+) -\I}{1-r} \right) \times \R^n.\]

	Now let us prove that $(A_r, v_r) \to (\I,0)$.
	By Propositions \ref{prop_Lcoercive} and \ref{prop_LIdbounded}, the sequence $(A_r, v_r)$ is bounded.
	Assume $(A_r, v_r)$ does not converge to $(\I,0)$, then there is a sequence $r_k \to 1^-$ such that $(A_{r_k}, v_{r_k}) \to (A^*, v^*) \in \SL \times \R^n$, with $(A^*, v^*) \neq (\I,0)$.
	Since the \Lowner position is unique up to orthogonal transformations and $A^* \in \symp$, the set $(A^*K+v^*) \setminus \B$ has positive Lebesgue measure.
	By Fatou's lemma,
	\begin{align}
		\liminf_{k \to \infty} L_r(A_{r_k}, v_{r_k}) 
		&\geq \int_{\R^n \setminus \B} \liminf_{k \to \infty} \frac 1{1-r_{k}} f_{r_k}(|x|_2) g_{r_k}(\|x-v_{r_k}\|_{A_{r_k}K}) dx \\
		&\geq \int_{(A^*K+v^*) \setminus \B} \liminf_{k \to \infty} \frac 1{1-r_{k}} f_{r_k}(|x|_2) dx \\
		&= \infty
	\end{align}
	which is absurd because by minimality and Proposition \ref{prop_LIdbounded},
	\[
		L_r(A_{r_k}, v_{r_k}) \leq L_r(\I,0) \leq C.
	\]

	It remains to show that $\tr\left(\frac{M_r}{\|M_r\|_F}\right) \to 0$.
	Recall that the trace is the differential of $\det$ at $\I \in \M$ by Taylor,
	\begin{equation}
		\det(\I + V) = 1 + \tr(V) + o(\|V\|)
	\end{equation}
	where $\frac{o(\varepsilon)}{\varepsilon} \to 0$ as $\varepsilon \to 0$.
	Taking $V = (1-r) M_r$ we get
	\begin{align}
		1
		&= \det(A_r)\\
		&= \det(\I + (1-r) M_r)\\
		&= 1 + (1-r) \tr(M_r) + o((1-r) \|M_r\|_F)
	\end{align}
	that shows the result.
\end{proof}

\begin{proof}[Proof of Theorem \ref{thm_r1extension}]
	Write the function $I_r$ in polar coordinates to obtain
	\begin{align}
		I_r(M,v)
		&= \frac 1{1-r} \int_{\R^n} f_r(|x|_2) g_r(\|(\I+(1-r)M)^{-1}.(x-(1-r)v)\|_K) dx\\
		&= \frac 1{1-r} \int_\S \int_{0}^\infty f_r(s) g_r(\|(\I+(1-r)M)^{-1}.(s\xi-(1-r)v)\|_K) \\ & \times s^{n-1} ds d\xi.\\
	\end{align}
	Since $K$ is smooth, by Taylor expansion and formula \eqref{eq_normgradient} we obtain for any $x,v \in \R^n$,
	\begin{align}
		\label{eq_r1extension_taylor_norm}
		\|x+v\|_K
		&= \|x\|_K + \langle \nabla \|x\|_K, v \rangle + o(|v|_2)\\
		&= \|x\|_K + \left\langle \frac{\|x\|_K}{\langle n^K(x), x \rangle} n^K(x), v \right\rangle + o(|v|_2).
	\end{align}
	We will denote by $o((1-r)^a)$ (resp. $o(1)$) any function of the involved parameters $M,v,r,s,t,\xi$, satisfying
	\[\lim_{r \to 1^-} \frac{o((1-r)^a)}{(1-r)^a} = 0 \left(\text{resp.} \lim_{r \to 1^-} o(1) = 0 \right)\]
	where the limits are uniform in compact sets with respect to the parameters.
	Likewise, $O(1)$ will denote any bounded function.
	For any $\xi,v \in \R^n$, $s \geq r > 0, M \in B_r \subseteq \M$,
	\begin{align}
		\label{eq_r1extension_taylor_matrix}
		(\I + (1-r)M)^{-1}(s\xi - (1-r) v)
		&= (s \xi - (1-r) v) \\ & - (1-r) M (s \xi - (1-r) v) + o(1-r)\\
		&= s \xi - (1-r) (s M \xi + v) + o(1-r).\\
	\end{align}
	Using \eqref{eq_r1extension_taylor_matrix} and \eqref{eq_r1extension_taylor_norm},
	\begin{align}
		\label{eq_r1extension_taylor_composition}
		\left\| (\I + (1-r)M)^{-1} \right. & \left. (s\xi - (1-r) v) \right\|_K\\
		&= \left\| s \xi - (1-r) (s M \xi + v) + o(1-r) \right\|_K \\
		&= s \left\| \xi - (1-r) \left(M \xi + \frac 1s v\right) + o(1-r) \right\|_K\\
		&= s \left( \|\xi\|_K - (1-r) \left\langle \frac{\|\xi\|_K}{\langle n^K(\xi), \xi \rangle} n^K(\xi), M \xi + \frac 1s v \right\rangle \right) \\ & + o(1-r).
	\end{align}

	Putting all together and making the substitution $s = 1 + (1-r) t$ we get
	\begin{align}
		I_r(M,v)
		&= \frac 1{1-r} \int_\S \int_0^\infty f_r(s) \\
		&\times g_r\left( s \left( \|\xi\|_K - (1-r) \left\langle \frac{\|\xi\|_K}{\langle n^K(\xi), \xi \rangle} n^K(\xi), M \xi + \frac 1s v \right\rangle \right) + o(1-r) \right) s^{n-1} ds d\xi\\
		&= \int_\S \int_{-\frac 1{1-r}}^\infty f(t) \\
			&\times g\left( \frac{\|\xi\|_K - 1}{1-r} - \left\langle \frac{\|\xi\|_K}{\langle n^K(\xi), \xi \rangle} n^K(\xi), M \xi + v + o(1) \right\rangle + t (\|\xi\|_K + o(1)) + o(1) \right)\\
			&\times (1+(1-r)t)^{n-1} dt d\xi.\\
	\end{align}

	Notice that $\|\xi\|_K=1$ for $\xi \in \St$, $\|\xi\|_K>1$ for $\xi \in \S \setminus \partial K$ and that $n^K(\xi) = \xi$ for every $\xi \in \St$.
	We have $\lim_{r \to 1^-} \frac{\|\xi\|_K - 1}{1-r} \to \infty$ in $\S \setminus \partial K$ and since $0$ is in the interior of $K$, $\langle n^K(\xi), \xi \rangle$ and $\|\xi\|_K$ are bounded from below.
	Also, by \ref{f_zeroleft} the integrand is $0$ for $t<-1$, then
	\begin{align}
		I_r(M,v)
		&= \int_\St \int_{-1}^\infty f(t) g\left( - \left\langle \xi, M \xi + v + o(1) \right\rangle + t (1 + o(1)) + o(1) \right)\\
			&\times (1+(1-r)t)^{n-1} dt d\xi\\
		&+ \int_{\S \setminus \partial K} \int_{-1}^\infty f(t) g\left( \frac{\|\xi\|_K - 1}{1-r} + O(1) + t (\|\xi\|_K + o(1)) + o(1) \right)\\
			&\times (1+(1-r)t)^{n-1} dt d\xi.\\
	\end{align}

	To prove the uniform convergence in compact sets, consider a convergent sequence $(M_k, v_k) \to (M,v)$ and $r_k \to 1^-$.
	By property \ref{g_0}, the function $g$ in the second integral is zero for $t \geq C$ where $C$ is independent of $k$.
	The functions $f,g$ are thus uniformly bounded in the support of both integrals, and we may apply the Dominated Convergence Theorem to obtain (thanks to property \ref{g_smooth})
	\[\lim_{k \to \infty} I_{r_k}(M_k,v_k) = \int_\St f(t) g\left( -\left\langle \xi, M \xi + v \right\rangle + t \right) d\xi.\]

	Finally we show that $F$ satisfies the assumptions of Theorem \ref{thm_isotropic_general}.
	Only the convexity is non-trivial.
	By \ref{f_smooth},\ref{g_smooth}, $f$ and $g$ are absolutely continuous and differentiable a.e., $F$ is twice differentiable a.e. and
	\[F''(x) = \int_{-1}^1 \bar g'(t)f'(x-t) dt \geq 0,\]
	showing that $F$ is convex.
	To see the strict convexity in $[0,\infty)$ take any $x > 0$.
	If $F''(x) = 0$ the last inequality implies that $f'=0$ in a set of positive measure inside $(x-1,x+1)$ (namely, the set where $g' < 0$) and this contradicts \ref{f_inc}.
\end{proof}

\begin{proof}[Proof of Theorem \ref{thm_min_equivalences}]
\ \\
			\underline{$1 \Rightarrow 2$:}\\
			Assume by contradiction that for some $(M, w) \in (\sym_0 \times \R^n) \setminus (0,0)$, $\langle \xi, M \xi + w \rangle \leq 0$ for $\nu$-almost every $\xi$.
			For $\lambda > 1$, since $F$ is non-decreasing,
			\begin{align}
				I_\nu(\lambda(M,w)) 
				&= \int_\St F(\lambda \langle\xi, M \xi + w \rangle) d\nu \\
				&\leq \int_\St F(\langle\xi, M \xi + w \rangle) d\nu \\
				&= I_\nu(M,w).
			\end{align}
			If $(M_0, w_0)$ is the global minimum, by convexity of $I_\nu$
			\[I_\nu\left(\left(1-\frac 1\lambda \right) (M_0, w_0) + \frac 1\lambda \lambda(M,w) \right) \leq \left(1-\frac 1\lambda \right) I_\nu(M_0, w_0) + \frac 1\lambda I_\nu(\lambda(M,w)).\]
			Taking limits when $\lambda \to +\infty$ we obtain
			\[I_\nu\left( (M_0, w_0) + (M,w) \right) \leq I_\nu(M_0, w_0)\]
			which is absurd.

			\underline{$2 \Rightarrow 1$:}\\
			Consider
			\[E_{(M,w)} = \left\{\xi \in \St / \langle \xi, M \xi + w \rangle > 0\right\}.\]
			We claim that $\varepsilon := \inf \nu(E_{(M,w)}) > 0$, where the infimum is taken over all \\ ${(M, w) \in \sym_0 \times \R^n}$ with $\|(M,w)\|=1$.

			Assume by contradiction that there is a convergent sequence $(M_k, w_k) \to (M_0, w_0)$ for which $\nu(E_{(M_k, w_k)}) \to 0$.
			If $\xi \in E_{(M_0, w_0)}$, then $\langle \xi, M_0 \xi + w_0 \rangle > 0$ and there is $k_0$ for which $k > k_0$ implies $\xi \in E_{(M_k, w_k)}$.
			This means that $E_{(M_0, w_0)} \subseteq \liminf\{E_{(M_k, w_k)}\}$ and we have
			\[\nu(E_{(M_0, w_0)}) \leq \nu(\liminf\{E_{(M_k, w_k)}\}) \leq \liminf \nu(E_{(M_k, w_k)}) = 0\]
			which is absurd since $\|(M_0, w_0)\| = 1$.

			Using the comparison $F(x) \geq F'(0) x_+$, we deduce that 
			\begin{align}
				I_\nu(M,w)
				&\geq \int_\S F'(0) \langle \xi, M \xi + w \rangle_+ d\nu\\
				&\geq \|(M, w)\| F'(0) \varepsilon \nu(\S) \\
			\end{align}
			which implies the coercivity of $I_\nu$.
			
			Now take $(M_0, w_0), (M_1, w_1)$ and $(M_\lambda, w_\lambda) = \lambda (M_0, w_0) + (1-\lambda) (M_1, w_1)$ with $\lambda \in [0,1]$.
			Denote momentarily $\xi^{(\lambda)} = \langle \xi, M_\lambda \xi + w_\lambda \rangle$, for short.
			Since $F$ is convex in $\R$ and strictly convex in $[0,\infty)$ we have
			\begin{align}
				\int_{\xi^{(\lambda)} > 0} F\left(\xi^{(\lambda)} \right) d\nu
				&< \int_{\xi^{(\lambda)} > 0} \lambda F\left(\xi^{(0)} \right) + (1-\lambda) F\left( \xi^{(1)} \right) d\nu\\
				\int_{\xi^{(\lambda)} \leq 0} F\left(\xi^{(\lambda)} \right) d\nu
				&\leq \int_{\xi^{(\lambda)} \leq 0} \lambda F\left(\xi^{(0)} \right) + (1-\lambda) F\left( \xi^{(1)} \right) d\nu.\\
			\end{align}

			We obtain
			\[
				I_\nu(M_\lambda, w_\lambda) < \lambda I_\nu(M_0, w_0) + (1-\lambda) I_\nu(M_1, w_1),
			\]
			thus $I_\nu$ is strictly convex and coercive, and it has a unique global minimum.

			\underline{$3 \Rightarrow 2$, case $\nu = \nu_K$:}\\
			Let $(M,w) \in \sym_0 \times \R^n$. By Theorem \ref{thm_lowner_equivalences} there exists $\xi_0 \in \St$ such that $\langle \xi_0, M \xi_0 + w \rangle > 0$.
	Now we may find $\varepsilon > 0$ such that $|\xi - \xi_0|_2 < \varepsilon$ implies $\langle \xi, M \xi + w \rangle > 0$ as well.
	Since $\xi_0$ is in the support of $\nu_K$, $\nu_K(B(\xi_0, \varepsilon)) > 0$ (Proposition \ref{thm_Csupport}).
			This implies that
			\[\nu_K \left(\{ \xi \in \S / \langle \xi, M \xi + w \rangle > 0 \}\right) > 0.\]

			\underline{$2 \Rightarrow 3$, case $\nu = \nu_K$:}\\
			Is trivial since $\nu_K(\{\xi \in \St / \langle \xi, M \xi + w \rangle > 0\})>0$ implies $(\{\xi \in \St / \langle \xi, M \xi + w \rangle > 0\})$ is non-empty.

		\underline{$2 \Leftrightarrow 3$, case $\nu = c$:}\\
			Since $\St$ is finite, $\co(\St)$ is a polytope and $\nu_K$ is an atomic measure supported in $\St$ (Proposition 11 again). The equivalence then follows from the case $\nu = \nu_K$.
\end{proof}

\begin{proof}[Proof of Theorem \ref{thm_isotropic_general}]
	We consider the inner product in $\M \times \R^n$ mentioned in the introduction.
	It is trivial to verify that
	that $\sym_0 \times \R^n$ is the orthogonal complement of $(\I,0)$ in $\sym \times \R^n$.
	We compute the derivative of $I_\nu$ in the direction of $(\bar M, \bar v)$, at the point $(M,v)$:
	\begin{align}
		\langle \nabla I_\nu(M,v), (\bar M, \bar v) \rangle
		&= \int_\S F'(\langle \xi, M \xi + v \rangle) \langle \xi, \bar M \xi + \bar v \rangle d\nu\\
		&= \int_\S F'(\langle \xi, M \xi + v \rangle) (\langle \xi \otimes \xi, \bar M \rangle + \langle \xi, \bar v \rangle ) d\nu\\
		&= \left\langle \int_\S F'(\langle \xi, M \xi + v \rangle) \left(\xi \otimes \xi, \xi \right) d\nu, (\bar M, \bar v) \right\rangle\\
	\end{align}
	we deduce that
	\[
		\nabla I_\nu(M,v)
		= \int_\S F'(\langle \xi, M \xi + v \rangle) \left(\xi \otimes \xi, \xi \right) d\nu
	\]
	since this expression already in $\sym \times \R^n$.

	The gradient of the function $T(M,v) = \tr(M)$ is $\nabla T(M,v) = (\I, 0)$.
	Then the equation of Lagrange multipliers at a minimum $(M_0, w_0)$ is written as
	\[ \int_\S F'(\langle \xi, M_0 \xi + w_0 \rangle) \left(\xi \otimes \xi, \xi \right) d\nu = \lambda (\I,0).\]
	This clearly implies that
	\begin{equation}
		\label{eq_isotropic_general_1}
		\int_\S F'(\langle \xi, M_0 \xi + w_0 \rangle) \left(\xi \otimes \xi\right) d\nu = \lambda\ \I
	\end{equation}
	\begin{equation}
		\int_\S F'(\langle \xi, M_0 \xi + w_0 \rangle) \xi d\nu = 0.
	\end{equation}
	Since $F$ is non-decreasing, $F'(\langle \xi, M_0 \xi + w_0 \rangle) \geq 0$.
	Taking traces in equation \eqref{eq_isotropic_general_1} we get
	\[\lambda = \frac 1n \int_\S F'(\langle \xi, M_0 \xi + w_0 \rangle) d\nu. \]
	By Theorem \ref{thm_min_equivalences}, we know that $\langle \xi, M_0 \xi + w_0 \rangle > 0$ for a set of positive $\nu$-measure.
	Since $F'(x) \geq 0$ for every $x$ and $F'(x) > 0$ for $x \geq 0$, we deduce that $\lambda > 0$ and the proof is complete.
\end{proof}

\begin{proof}[Proof of Theorem \ref{thm_ADerivative}]
	First we prove that if $I_1$ has a unique global minimum $(M_0, w_0)$ then $(M_r, w_r)$ converges to $(M_0, w_0)$.
	By coercivity of $I_1$ there exists $R > 0$ such that $(M,v) \in \sym_0 \times \R^n, \left\| (M,v) \right\| \geq R$ implies
	\[I_1(M,w) \geq C+2\]
	where $C$ is given in Proposition \ref{prop_LIdbounded}.

	Let $B_R=\left\{(M,w) \in \sym \times \R^n/ \left\| (M,w)\right\| \leq R\right\}$. 
	Since $I_1$ is continuous in the compact set $B_R$, there is $\varepsilon > 0$ such that 
	\[I_1(M,w) \geq C + 1\]
	for every $(M,w) \in \partial B_R$ with $|\tr(M)| < \varepsilon$.
	By Theorem \ref{thm_r1extension}, there is $r_0 \in (1/2,1)$ such that for every $r \in (r_0, 1)$ and $(M,w) \in B_R$,
	\[|I_r(M,w) - I_1(M,w)| \leq 1/2.\]
	Increasing $r_0$ if necessary, we may assume that for every $r \in (r_0, 1)$ and $\lambda \in [0,1]$,
	\[\det(\lambda (A_r, v_r) + (1-\lambda) (\I,0)) \leq \frac {C+1/2}{C+1/4} = 1 + \frac{1}{4C+1}\]
	and that $\left|\tr\left( \frac{M_r}{\|M_r\|_F}\right)\right| \leq \frac{\varepsilon}R$ (last part of Theorem \ref{thm_min_existence}).

	First we claim that $(M_r,w_r) \in B_R$ for $r \in (r_0,1)$.
	Assume by contradiction that $(M_r,w_r) \not\in B_R$ for some $r \in (r_0,1)$, and consider $\lambda < 1$ such that $\|\lambda (M_r,w_r) \| = R$, then since $\left| \tr(\lambda M_r)\right| \leq \frac{\|\lambda M_r\|_F}{R} \varepsilon \leq \varepsilon$,
	\[I_r(\lambda (M_r,w_r)) \geq I_1(\lambda (M_r,w_r)) - \frac 12 \geq C + 1 - 1/2.\]
	By the convexity of $\mathcal D$,
	\[(\I,0) + (1-r) \lambda (M_r, w_r) = \lambda (A_r,v_r) + (1-\lambda) (\I,0) \in \mathcal D\]
	\begin{align}
		L_r(\lambda (A_r,v_r) + (1-\lambda) (\I, 0)) 
		&= \det(\lambda (A_r, v_r) + (1-\lambda) (\I,0))^{-1} I_r(\lambda (M_r,w_r)) \\
		&\geq \left(\frac {C+1/2}{C+1/4}\right)^{-1} ( C + 1/2) = C+1/4\\
		&\geq L_r(\I,0) + 1/4\\
		&\geq L_r(A_r,v_r) + 1/4.
	\end{align}
	We obtain the inequalities
	\[L_r(\I,0) < L_r(\lambda (A_r,v_r) + (1-\lambda) (\I,0)) > L_r(A_r,v_r)\]
	contradicting the convexity of $L_r$, and the claim is proved.

	For any matrix $M \in \sym_0$ consider 
	\[M^{(r)} = \frac{\det(\I + (1-r)M)^{-1/n}(\I + (1-r) M) - \I}{1-r}\]
	and notice that it belongs to $\frac{(\SL\cap \sym_+) -\I}{1-r}$ for $r$ close to $1$.
	We also have
	\begin{align}
		\lim_{r \to 1^-} M^{(r)}
		&= \lim_{r \to 1^-}\frac{\det(\I + (1-r)M)^{-1/n} - 1}{1-r} \I + \det(\I + (1-r)M)^{-1/n} M\\
		&= \left. \frac{\partial}{\partial t}\right|_{t=0}\left(\det(\I + t M)^{-1/n} \I \right) + M\\
		&= -\frac 1n \tr(-M) \I + M\\
		&=M.
	\end{align}

	Now that $(M_r,w_r)$ is bounded, for every convergent sequence $(M_{r_k},w_{r_k}) \to (M_0, w_0)$ with $r_k \to 1^-$, and for every $(M,w) \in \sym_0 \times \R^n$,
	\[I_1(M_0, w_0) \leftarrow I_{r_k}(M_{r_k}, w_{r_k}) \leq I_{r_k}(M^{(r_k)},w) \to I_1(M,w)\]
	so that $(M_0,w_0)$ is the (unique) minimum of $I_1$, and we deduce $(M_r, w_r) \to (M_0, w_0)$ as desired.

	Finally, we write
	\begin{align}
		\left. \frac{\partial (A_r, v_r)}{\partial r} \right|_{r=1}
		&= \lim_{r\to 1^-} \frac{(A_r, v_r) - (\I,0)}{r-1}\\
		&= \lim_{r\to 1^-} (-M_r,-w_r)\\
		&= -(M_0,w_0).
	\end{align}

	For the second part of the theorem take $\delta$ any continuous function with compact support and write, as in the proof of Theorem \ref{thm_r1extension},
	\begin{align}
		\frac 1{1-r} \int_{\R^n} \delta(x) \frac{(f')_r(|x|_2)}{|x|_2} & g_r(\|\bar A_r^{-1}(x-\bar v_r)\|_K) dx\\
		&= \int_\St \int_{-1}^\infty \delta((1+(1-r)t) \xi ) f'(t) \\ &\times g\left( - \left\langle \xi, M_0 \xi + w_0 + o(1) \right\rangle + t (1 + o(1)) + o(1) \right)\\
			&\times (1+(1-r)t)^{n-2} dt d\xi\\
		&+ \int_{\S \setminus \partial K} \int_{-1}^\infty \delta((1+(1-r)t) \xi) f'(t) \\ &\times g\left( \frac{\|\xi\|_K - 1}{1-r} + O(1) + t (\|\xi\|_K + o(1)) + o(1) \right)\\
			&\times (1+(1-r)t)^{n-2} dt d\xi\\
	\end{align}
	hence by the Dominated Convergence Theorem,
	\[ \frac 1{1-r} \int_{\R^n} \delta(x) \frac{(f')_r(|x|_2)}{|x|_2} g_r(\|A_r^{-1}(x-v_r)\|_K) dx \to \int_\St \delta(\xi) F'(\langle \xi, M_0 \xi + v_0 \rangle d\xi. \]
\end{proof}

\bibliographystyle{abbrv}
\bibliography{ref}

\end{document}